\documentclass[11pt]{article}
\usepackage{article}

\author{Torgeir Aamb\o}
\title{Algebraicity in monochromatic homotopy theory}

\date{}

\begin{document}
\maketitle

\begin{abstract}
    Using Patchkoria--Pstr{\k a}gowski's version of Franke's algebraicity theorem, we prove that the category of $K_p(n)$-local spectra is exotically equivalent to the category of derived $I_n$-complete periodic comodules over the Adams Hopf algebroid $(E_*, E_*E)$ for large primes. This gives a finite prime result analogous to the asymptotic algebraicity for $\spK$ of Barthel--Schlank--Stapleton. 
\end{abstract}


\tableofcontents

\section{Introduction}
\label{sec:introduction}

The central idea in chromatic homotopy theory is to study the symmetric monoidal stable $\infty$-category of spectra, $\sp$, via its smaller building blocks. These are the categories $\sp\np$ and $\spK$ of $E\np$-local and $K_p(n)$-local spectra, where $E=E\np$ is Morava $E$-theory, and $K_p(n)$ is Morava $K$-theory, see for example \cite{hovey-strickland_99}. These categories depend on a prime $p$ and an integer $n$, called the height. For a fixed height $n$, increasing the prime $p$ makes both categories behave more algebraically. This manifests itself, for example, in the $E$-Adams spectral sequence of signature
$$E^{s,t}_2 (L_n\S) = \Ext_{E_*E}^{s,t}(E_*,E_*)\Longrightarrow \pi_{t-s} L_n\S$$
computing the homotopy groups of the $E$-local sphere. By the smash product theorem of Ravenel, see \cite[7.5.6]{ravenel_92}, this spectral sequence has a horizontal vanishing line at a finite page. If $p>n+1$, this vanishing line appears already on the second page, where the information is completely described by the homological algebra of $\comod\EE$---the Grothendieck abelian category of comodules over the Hopf algebroid $(E_*, E_*E)$. 

Increasing the prime $p$ correspondingly increases the distance between objects appearing in the $E$-Adams spectral sequence. When $2p-2$ exceeds $n^2+n$, there is no longer room for any differentials, and the  spectral sequence in fact collapses to an isomorphism
$$\pi_* L_n\S\cong \Ext^{*,*}_{\EE}(E_*, E_*),$$
for degree reasons. In other words, the homotopy groups are completely algebraic in this range. 

A natural question to ask is whether this collapse is a feature solely of the $E$-Adams spectral sequence or if it is a feature of the category $\sp\np$. More precisely, is the entire category of $E$-local spectra algebraic, in the sense that it is equivalent to a derived category of an abelian category, whenever $2p-2>n^2+n$? 

At height $n=0$, the category $\sp\np$ is the category of rational spectra $\sp_\Q$, which can be seen to be equivalent to the derived $\infty$-category of rational vector spaces, but at positive heights $n>0$, there can never be an equivalence of $\infty$-categories $\sp\np \simeq D(\A)$. 

However, in \cite{bousfield_1985} Bousfield showed that for $p>2$ and $n=1$, that there is an equivalence of homotopy categories
$$h \sp_{1,p}\simeq h \Fr_{1,p},$$
where $\Fr\np$ is a certain derived $\infty$-category of twisted comodules over $(E_*, E_*E)$. As this cannot be lifted to an equivalence of $\infty$-categories, it is sometimes referred to as an \emph{exotic} equivalence. 

Franke expanded upon this in \cite{franke_96} by conjecturing---and attempting to prove---that for $2p-2 > n^2+n$ there should be an equivalence of homotopy categories
$$h \sp\np\simeq h \Fr\np.$$
Unfortunately, a subtle error was discovered in the proof by Patchkoria in \cite{patchkoria_2013}, but the result was recovered in \cite{pstragowski_2021} with a slightly worse bound: $2p-2>2n^2+2n$. Pstr{\k a}gowski also proved that this equivalence gets ``stronger'' the larger the prime, where we not only get an equivalence of categories but an equivalence of $k$-categories 
$$h_k \sp\np\simeq h_k \Fr\np,$$
for $k=2p-2-n^2-n$. Here $h_k \C$ denotes taking the homotopy $k$-category, given by $(k-1)$-truncating the mapping spaces in $\C$. At $k=1$, this gives the classical situation of taking the homotopy category $h\C$. Using and developing a more general machinery, Pstr{\k a}gowski and Patchkoria proved in \cite{patchkoria-pstragowski_2021} that the above equivalence holds in Franke's conjectured bound, $2p-2>n^2+n$.

These results imply that increasing the prime $p$ decreases how destructive the $k$-truncation of the mapping spaces needs to be. In the limit $p\rightarrow \infty$, we might expect that there is no need to truncate at all, giving an equivalence of $\infty$-categories. But, there needs to be an appropriate notion of what ``going to the infinite prime'' should be. In \cite{barthel-schlank-stapleton_2020}, the authors use a notion of ultraproducts over a non-principal ultrafilter $\mathcal{F}$ of primes to formalize this limiting process. They use this to prove the existence of a symmetric monoidal equivalence of $\infty$-categories
$$\prod_{\mathcal{F}}\sp\np \simeq \prod_{\mathcal{F}}\Fr\np.$$ 
Expanding on their work, Barthel, Schlank, and Stapleton proved in \cite{barthel-schlank-stapleton_2021} a $K_p(n)$-local version of the above result. More precisely, they show that there is a symmetric monoidal equivalence of $\infty$-categories
$$\prod_{\mathcal{F}}\spK \simeq \prod_{\mathcal{F}}\Frc\np,$$
where the right-hand side consists of derived complete twisted comodules for the naturally occurring Landweber ideal $I_n\subseteq E_*$.

\subsection*{Statement of results}

We can summarize the most general of the above algebraicity results in the following table,
\begin{table}[h]
    \centering
    \begin{tabular}{c|ccc}
        & $p<\infty $ & $p\rightarrow \infty$ \\
        \hline 
        $\sp\np$& \cite{patchkoria-pstragowski_2021} & \cite{barthel-schlank-stapleton_2020} \\
        $\spK$ &  & \cite{barthel-schlank-stapleton_2021} 
    \end{tabular}
\end{table}

A natural question arises: Is there a finite prime exotic algebraicity for $\spK$? The goal of this paper is to give an affirmative answer. More precisely, we prove the following. 

\begin{introthm}[\cref{thm:main-spectra-dual}]
    \label{thm:A}
    Let $p$ be a prime and $n\in \N$. If $k=2p-2-n^2-n>0$,  then there is an equivalence of $k$-categories 
    $$h_k \spK\simeq h_k \Frc\np.$$
    In other words, $K_p(n)$-local spectra are exotically algebraic at large primes. 
\end{introthm}

The available tools for proving such a statement require an abelian category with enough injective objects admitting lifts to a stable $\infty$-category. In lack of such a well-behaved abelian approximation for $\spK$, we take inspiration from \cite{barthel-schlank-stapleton_2021} and instead use the dual category $\M\np$ of monochromatic spectra, which we show has the needed properties. \cref{thm:A} will then follow from the following result. 

\begin{introthm}[\cref{thm:main-spectra}]
    \label{thm:B}
    Let $p$ be a prime and $n\in \N$. If $k=2p-2-n^2-n>0$,  then there is an equivalence of $k$-categories $h_k \M\np\simeq h_k \Fr\np^{I_n-tors}.$
\end{introthm}

In order to prove \cref{thm:B}, we first prove the analogous statement for monochromatic $E$-modules. 
    
\begin{introthm}[\cref{thm:main-modules}]
    \label{thm:C}
    Let $p$ be a prime and $n\in \N$. If $k=2p-2-n>0$,  then there is an equivalence of $k$-categories $h_k \Modt\simeq h_k D^{per}(\modt).$
\end{introthm}

\subsection*{Overview of the paper}

\cref{sec:introduction} introduces local duality, and the proposed exotic algebraic model using periodic chain complexes of torsion comodules. \cref{sec:exotic-algebraic-models} focuses on Franke's algebraicity theorem. Most of the new results of the paper are presented in \cref{ssec:algebraicity-modules} and \cref{ssec:algebraicity-spectra}, where we prove \cref{thm:A}, \cref{thm:B} and \cref{thm:C}. In \cref{app:barr-beck} we prove that Barr-Beck adjunctions interact well with local duality, which is used to prove that periodization, torsion and taking the derived category all commute.

\textbf{Acknowledgements.} We want to thank Drew Heard, Irakli Patchkoria and Marius Nielsen for helpful conversations and for proof-reading the paper. We also want to thank Piotr Pstr\a{}gowski for finding a mistake in the proof of a previous version of \cref{lm:cohomological-dimension-torsion-comodules}. Lastly, we want to thank the University of Copenhagen for their hospitality while writing most of this paper. This work forms a part of the authors thesis, partially supported by grant number TMS2020TMT02 from the Trond Mohn Foundation. 

\section{The algebraic model}

The goal of this section is to set up the necessary background material that will be used throughout the paper. We use these to construct convenient algebraic approxomations of categories arising from chromatic homotopy theory. 

\subsection*{Some conventions}

We freely use the language of $\infty$-categories, as developed by Joyal \cite{joyal_02} and Lurie \cite{lurie_09, Lurie_HA}. Even though we are dealing with both classical $1$-categories and $\infty$-categories in this paper, we will sometimes refer to them both as \emph{categories}, hoping that the prefix is clear from the context. 

We denote by {\defn $\Pr$} the $\infty$-category of presentable stable $\infty$-categories and colimit preserving functors. Together with the Lurie tensor product, it is a symmetric monoidal $\infty$-category. The category of algebras $\Alg(\Pr)$ is then the category of presentable stable $\infty$-categories with a symmetric monoidal structure commuting with colimits separately in each variable. 

Let $\C, \D\in \Alg(\Pr)$. A localization is a functor $f\colon \C\to \D$ with a fully faithful right adjoint $i$. We denote the composite by $L= i\circ f$. The adjoint $i$ identifies $\D$ with a full subcategory of $\C$, which we denote by $\C_L$. We then view $L$ as a functor $L\colon \C\to\C_L$, that is left adjoint to the inclusion, and by abuse of notation also call these localizations.


\subsection{Local duality}

The theory of abstract local duality, proved in \cite{hovey-palmiery-strickland_97} and generalized to the $\infty$-categorical setting in \cite{barthel-heard-valenzuela_2018} will be important for the entire paper. In particular, it is the technology that will allow us to translate \cref{thm:B} into \cref{thm:A}. 

\begin{definition}
    \label{def_local-duality-context}
    A pair $(\C, \K)$, where $\C\in \Alg(\Pr)$ is compactly generated by dualizables, and $\K$ is a subset of compact objects, is called a {\defn local duality context}.
\end{definition}

\begin{construction}
    Let $(\C, \K)$ be a local duality context. We define {\defn $\C^{\K-tors}$} to be the localizing tensor ideal generated by $\K$, denoted $\Loc^\otimes_\C(\K)$. Further we define {\defn $\C^{\K-loc}$} to be the left orthogonal complement $(\C^{\K-tors})^\perp$, i.e., the full subcategory consisting of objects $C\in \C$ such that $\Hom_\C(T,C)\simeq 0$ for all $T\in \C^{\K-tors}$. Similarily we define {\defn $\C^{\K-comp}$} to be the double left-orthogonal complement $(\C^{\K-loc})^\perp$. These full subcategories are respectively called the $\K$-torsion, $\K$-local and $\K$-complete objects in $\C$. We have inclusions into $\C$, denoted {\defn $i_{\K-tors}$}, {\defn $i_{\K-loc}$} and {\defn $i_{\K-comp}$} respectively. When the $\K$ is understood, we sometimes omit it from the notation. 
    
    By the adjoint functor theorem, \cite[5.5.2.9]{lurie_09}, the inclusions $i_{\K-loc}$ and $i_{\K-comp}$ have left adjoints {\defn $L_\K$} and {\defn $\Lambda_\K$} respectively, while $i_{\K-tors}$ and $i_{\K-loc}$ have right adjoints {\defn $\Gamma_\K$} and {\defn $V_\K$} respectively. These are then, by definition, localizations and colocalizations. Since the torsion, local and complete objects are ideals, these localizations and colocalizations are compatible with the symmetric monoidal structure of $\C$, in the sense of \cite[2.2.1.7]{Lurie_HA}. In particular, by \cite[2.2.1.9]{Lurie_HA} we get unique induced symmetric monoidal structures such that $\Gamma$, $L$ and $\Lambda$ are symmetric monoidal functors. 

    For any $X\in \C$, these functors assemble into two cofiber sequences:
    $$\Gamma_\K X \longrightarrow X \longrightarrow L_\K X \quad \text{and}\quad V_\K X \longrightarrow X \longrightarrow \Lambda_\K X.$$
    Note also that these functors only depend on the localizing subcategory $\C^{\K-tors}$, not on the particular choice of generators $\K$. Thus, when the set $\K$ is clear from the context, we often omit it as a subscript when writing the functors. 
\end{construction}

The following theorem is a slightly restricted version of the abstract local duality theorem of \cite[3.3.5]{hovey-palmiery-strickland_97} and \cite[2.21]{barthel-heard-valenzuela_2018}.  

\begin{theorem}
    \label{thm:local-duality}
    Let $(\C, \K)$ be a local duality context. Then
    \begin{enumerate}
        \item the functors $\Gamma$ and $L$ are smashing, i.e. $\Gamma X\simeq X\otimes \Gamma \1$ and $LX\simeq X\otimes L\1$,
        \item the functors $\Lambda$ and $V$ are cosmashing, i.e. $\Lambda X\simeq \iHom(\Gamma \1,X)$ and $VX\simeq \iHom(L\1, x)$, and 
        \item the functors $\Gamma\colon \C^{\K-comp}\longrightarrow \C^{\K-tors}$ and $\Lambda\colon \C^{\K-tors}\longrightarrow \C^{\K-comp}$ are mutually inverse symmetric monoidal equivalences of categories,
    \end{enumerate}
    This can be summarized by the following diagram of adjoints
    \begin{center}
        \begin{tikzcd}
                & {\C^{\K-loc}} \\
                & {\C} \\
                {\C^{\K-tors}} && {\C^{\K-comp}}
                \arrow["L", xshift=-2pt, from=2-2, to=1-2]
                \arrow[xshift=2pt, from=1-2, to=2-2]
                \arrow["\Lambda", yshift=2pt, xshift=2pt, from=2-2, to=3-3]
                \arrow[yshift=-2pt, xshift=-1pt, from=3-3, to=2-2]
                \arrow["\Gamma", yshift=-2pt, xshift=2pt, from=2-2, to=3-1]
                \arrow[yshift=2pt, xshift=-1pt, from=3-1, to=2-2]
                \arrow[bend left=35, dashed, from=3-1, to=1-2]
                \arrow[bend left=35, dashed, from=1-2, to=3-3]
                \arrow["\simeq"', swap, from=3-1, to=3-3]
        \end{tikzcd}    
    \end{center}
\end{theorem}

\begin{remark}
    \label{rm:monoidal-structure-in-local-duality}
    This implies, in particular, that the symmetric monoidal structure induced by the localization $L$ and the colocalization $\Gamma$ is just the symmetric monoidal structure on $\C$ restricted to the full subcategories. This is not the case for $\C^{\K-comp}$, where the symmetric monoidal structure is given by $\Lambda_\K(-\otimes_\C-)$. 
\end{remark}

We have two main examples of interest for this paper. 

\begin{example}
    \label{ex:local-duality-comod}
    Let $(A,\Psi)$ be an Adams type Hopf algebroid, for example the Hopf algebroid $(R_*, R_*R)$ for an Adams type ring spectrum $R$---see \cite[A.1]{ravenel_86} and \cite{hovey_04} for details. Denote by $D(\Psi)$ the derived $\infty$-category associated to the symmetric monoidal Grothendieck abelian category $\Comod_\Psi$. This is defined using the model structure from \cite{barnes-roitzheim_2011}. If $I\subseteq A$ is a finitely generated invariant regular ideal, then $(D(\Psi), A/I)$ is a local duality context, with associated local duality diagram
    \begin{center}
        \begin{tikzcd}
            & {D(\Psi)^{I-loc}} \\
            & {D(\Psi)} \\
            {D(\Psi)^{I-tors}} && {D(\Psi)^{I-comp}}
            \arrow["L_I^\Psi", xshift=-2pt, from=2-2, to=1-2]
            \arrow[xshift=2pt, from=1-2, to=2-2]
            \arrow["\Delta_I^\Psi", yshift=2pt, xshift=2pt, from=2-2, to=3-3]
            \arrow[yshift=-2pt, xshift=-1pt, from=3-3, to=2-2]
            \arrow["\Gamma^\Psi_I", yshift=-2pt, xshift=2pt, from=2-2, to=3-1]
            \arrow[yshift=2pt, xshift=-1pt, from=3-1, to=2-2]
            \arrow[bend left=35, dashed, from=3-1, to=1-2]
            \arrow[bend left=35, dashed, from=1-2, to=3-3]
            \arrow["\simeq"', swap, from=3-1, to=3-3]
        \end{tikzcd}    
    \end{center}
\end{example}

In \cref{ssec:the-algebraic-model} we compare $D(\Psi)^{I-tors}$ to a more concrete category: the derived category of $I$-power torsion comodules.  

The following example comes from chromatic homotopy theory. For a good introduction, see \cite{barthel-beaudry_19}. 

\begin{example}
    \label{ex:local-duality-chromatic}
    Let $E$ denote Morava $E$-theory at prime $p$ and height $n$. If $F(n)$ is a finite type $n$ spectrum, then the pair $(\sp\np, L_n F(n) )$ is a local duality context. The corresponding diagram can be recognized as
    \begin{center}
        \begin{tikzcd}
                & {\sp_{n-1,p}} \\
                & {\sp\np} \\
                {\M\np} && {\spK}
                \arrow["L_{n-1}", xshift=-2pt, from=2-2, to=1-2]
                \arrow[xshift=2pt, from=1-2, to=2-2]
                \arrow["L_{K_p(n)}", yshift=2pt, xshift=2pt, from=2-2, to=3-3]
                \arrow[yshift=-2pt, xshift=-1pt, from=3-3, to=2-2]
                \arrow["M\np", yshift=-2pt, xshift=2pt, from=2-2, to=3-1]
                \arrow[yshift=2pt, xshift=-1pt, from=3-1, to=2-2]
                \arrow[bend left=35, dashed, from=3-1, to=1-2]
                \arrow[bend left=35, dashed, from=1-2, to=3-3]
                \arrow["\simeq"', swap, from=3-1, to=3-3]
        \end{tikzcd}    
    \end{center}
    where $\M\np$ is the height $n$ monochromatic category and $\spK$ is the category of spectra localized at height $n$ Morava $K$-theory $K_p(n)$. The functor $L_{n-1}$ is the Bousfield localization at $E_{n-1}$, while $L_{K_p(n)}$ is the Bousfield localization at $K_p(n)$, see \cite{bousfield_1979_localization}. The local duality then exhibits the classical equivalence $\M\np\simeq \spK$, see \cite[6.19]{hovey-strickland_99}. 
\end{example}

\begin{remark}
    \label{rm:local-duality-modules}
    There is also a version of this local duality diagram for modules over $E$. This gives equivalences 
    \[\M\np\Mod_E\simeq \Mod_E^{I_n-tors}\simeq \Mod_E^{I_n-comp}\simeq L_{K_p(n)}\Mod_E,\]
    where $I_n$ is the Landweber ideal $(p,v_1, \ldots, v_{n-1})\subseteq E_*$.
\end{remark}


\subsection{The periodic derived torsion category}
\label{ssec:the-algebraic-model}

In this section we identify the category $D(\Psi)^{I-tors}$---as obtained in \cref{ex:local-duality-comod}---as the derived category of $I$-power torsion comodules. We also modify the category to exhibit some needed periodicity. 

\begin{definition}
    \label{def:I-power-torsion-comodule}
    Let $(A,\Psi)$ be an Adams Hopf algebroid and $I\subseteq A$ a regular invariant ideal. The $I$-power torsion of a comodule $M$ is defined as 
    $$T_I^\Psi M = \{x\in M \mid I^kx = 0 \text{ for some } k\in \N\}.$$
    We say a comodule $M$ is {\defn $I$-torsion} if the natural map $T_I^\Psi M\longrightarrow M$ is an equivalence. 
\end{definition}

\begin{remark}
    \label{rm:torsion-iff-underlying-is-torsion}
    One can similarily define $I$-power torsion $A$-modules. If $(A,\Psi)$ is an Adams Hopf algebroid, then a $\Psi$-comodule $M$ is $I$-power torsion if and only if its underlying module is $I$-power torsion, see \cite[5.7]{barthel-heard-valenzuela_2018}. 
\end{remark}

\begin{remark}
    \label{rm:torsion-comodules-grothendieck-monoidal}
    By \cite[5.10]{barthel-heard-valenzuela_2018} the full subcategory of $I$-torsion comodules, which we denote {\defn $\Comod_\Psi^{I-tors}$}, is a Grothendieck abelian category. It also inherits a symmetric monoidal structure from $\Comod_\Psi$. 
\end{remark}

The following technical lemma will be needed later. 

\begin{lemma}
    \label{lm:torsion-comodules-generated-by-compacts}
    Let $(A,\Psi)$ be an Adams Hopf algebroid, where $A$ is noetherian and $I\subseteq A$ a regular invariant ideal. Then $\Comod_\Psi^{I-tors}$ is generated under filtered colimits by the compact $I$-power torsion comodules. 
\end{lemma}
\begin{proof}
    By \cite[3.4]{barthel-heard-valenzuela_2020} $\Comod_\Psi^{I-tors}$ is generated by the set 
    $$\mathrm{Tors}_\Psi^{fp}:=\{G\otimes A/I^k \mid G \in \Comod_\Psi^{fp}, k\geq 1\},$$
    where $\Comod_\Psi^{fp}$ is the full subcategory of dualizable $\Psi$-comodules. Since $I$ is finitely generated and regular, $A/I^k$ is finitely presented as an $A$-module, hence it is compact in $\Comod_\Psi$ by \cite[1.4.2]{hovey_04}, and in $\Comod_\Psi^{I-tors}$ as colimits are computed in $\Comod_\Psi$. As $A$ is noetherian, being finitely generated and finitely presented coincide. The tensor product of finitely generated modules is finitely generated, hence any element in $\mathrm{Tors}_\Psi^{fp}$ is compact. 
\end{proof}

\begin{remark}
    The assumption that the ring $A$ is noetherian can most likely be removed, but it makes no difference to the results in this paper.  
\end{remark}

\begin{notation}
    Since $\Comod_\Psi^{I-tors}$ is Grothendieck abelian we have an associated derived stable $\infty$-category $D(\Comod_\Psi^{I-tors})$ which we denote simply by {\defn $D(\Psi^{I-tors})$}.
\end{notation}

We e can now compare the torsion category obtained from local duality and the derived category of $I$-power torsion comodules. 

\begin{lemma}[{\cite[3.7(2)]{barthel-heard-valenzuela_2020}}]
    \label{lm:derived-torsion-if-homology-torsion}
    Let $(A,\Psi)$ be an Adams Hopf algebroid and $I\subseteq A$ a regular invariant ideal. There is an equivalence of categories 
    $$D(\Psi)^{I-tors}\simeq D(\Psi^{I-tors}).$$ 
    Furthermore, an object $M\in D(\Psi)$ is $I$-torsion if and only if the homology groups $H_* M$ are $I$-power torsion $\Psi$-comodules.
\end{lemma}

In order to state both the general algebraicity matchinery of \cite{patchkoria-pstragowski_2021} and our results, we need the respective derived categories to exhibit the periodic nature of the spectra we are interested in. This is done via the periodic derived category. There are several ways to constructing this, but we follow \cite{franke_96} in spirit, using periodic chain complexes. 

\begin{definition}
    \label{def:periodic-chain-complex}
    Let $\A$ be an abelian category with a local grading, i.e., an autoequivalence $T\colon \A\to\A$, and denote $[1]$ the shift functor on the category of chain complexes $\Ch(\A)$ in $\A$. A chain complex $C\in \Ch(\A)$ is called {\defn periodic} if there is an isomorphism $\phi\colon C[1]\longrightarrow TC$. The full subcategory of periodic chain complexes is denoted by $\Ch^{per}(\A)$. 
\end{definition}

\begin{definition}
    The forgetful functor $\Ch^{per}(\A)\longrightarrow \Ch(\A)$ has a left adjoint $P$, called the {\defn periodization}. 
\end{definition}

\begin{definition}
    \label{def:periodic-derived-category}
    Let $\A$ be a locally graded abelian category. Then the {\defn periodic derived category} of $\A$, denoted {\defn $D^{per}(\A)$} is the $\infty$-category obtained by localizing $\Ch^{per}(\A)$ at the quasi-isomorphism. It is in fact stable by \cite[7.8]{patchkoria-pstragowski_2021}. 
\end{definition}

\begin{remark}
    \label{rm:periodic-derived-as-modules}
    If $\A$ is a symmetric monoidal category, then $P\1$ is a commutative ring object called the {\defn periodic unit}. The category of periodic chain complexes $\Ch^{per}(\A)$ is equivalent to $\Mod_{P\1}(\Ch(\A))$. This descends also to the derived categories, giving an equivalence 
    $$D^{per}(\A)\simeq \Mod_{P\1}(D(\A)),$$
    see for example \cite[3.7]{pstragowski_2021}. 
\end{remark}

We will also need local duality  for the periodic derived category associated to a Hopf algebroid. 

\begin{construction}
    \label{const:periodic-derived-local-duality}
    Let $(A, \Psi)$ be an Adams type (graded) Hopf algebroid. Then the shift functor $[1]\colon \Comod_\Psi\longrightarrow \Comod_\Psi$
    defined by $(TM)_k = M_{k-1}$ is a local grading on $\Comod_\Psi$. Denote the corresponding periodic derived category by $D^{per}(\Psi)$. The pair $(D^{per}(\Psi), P(A/I))$ is a local duality context with associated local duality diagram
    \begin{center}
    \begin{tikzcd}
        & {D^{per}(\Psi)^{I-loc}} \\
        & {D^{per}(\Psi)} \\
        {D^{per}(\Psi)^{I-tors}} && {D^{per}(\Psi)^{I-comp}}
        \arrow["L_I^\Psi", xshift=-2pt, from=2-2, to=1-2]
        \arrow[xshift=2pt, from=1-2, to=2-2]
        \arrow["\Lambda_I^\Psi", yshift=2pt, xshift=2pt, from=2-2, to=3-3]
        \arrow[yshift=-2pt, xshift=-1pt, from=3-3, to=2-2]
        \arrow["\Gamma^\Psi_I", yshift=-2pt, xshift=2pt, from=2-2, to=3-1]
        \arrow[yshift=2pt, xshift=-1pt, from=3-1, to=2-2]
        \arrow[bend left=35, dashed, from=3-1, to=1-2]
        \arrow[bend left=35, dashed, from=1-2, to=3-3]
        \arrow["\simeq"', swap, from=3-1, to=3-3]
    \end{tikzcd}    
    \end{center}
    The functors in the diagram are induced by the functors from \cref{ex:local-duality-comod}. In fact, there is a diagram 
    \begin{center}
        \begin{tikzcd}
            D(\Psi)^{I-tors} 
            \arrow[d, xshift=-2pt, "P", swap] 
            \arrow[r, yshift=2pt]       
            & D(\Psi) 
            \arrow[d, xshift=-2pt, "P", swap] 
            \arrow[r, yshift=2pt, "L_I^\Psi"]
            \arrow[l, yshift=-2pt, "\Gamma_I^\Psi"]       
            & D(\Psi)^{I-loc} 
            \arrow[d, xshift=-2pt, "P", swap] 
            \arrow[l, yshift=-2pt]      \\
            D^{per}(\Psi)^{I-tors} 
            \arrow[r, yshift=2pt]   
            \arrow[u, xshift=2pt] 
            & D^{per}(\Psi) 
            \arrow[r, yshift=2pt, "L_I^\Psi"]   
            \arrow[l, yshift=-2pt, "\Gamma_I^\Psi"] 
            \arrow[u, xshift=2pt] 
            & D^{per}(\Psi)^{I-loc} 
            \arrow[l, yshift=-2pt] 
            \arrow[u, xshift=2pt] 
        \end{tikzcd}    
    \end{center}
    that is commutative in all possible directions. Here the unmarked horizontal arrows are the respective fully faithful inclusions. 
\end{construction} 

\begin{remark}
    In the specific case of $(A, \Psi) = (E_0, E_0E)$ and $I\subseteq E_0$ the Landweber ideal $I_n$, then the above construction is \cite[3.12]{barthel-schlank-stapleton_2021}. 
\end{remark}

There is now some ambiguity to take care of for our category of interest $D^{per}(\Psi)^{I-tors}$. In the picture above, we do mean that we take $I$-torsion objects in $D^{per}(\Psi)$, i.e., $[D^{per}(\Psi)]^{I-tors}$, but we could also take the periodization of the category $D(\Psi^{I-tors})$ as our model. Luckily, there is no choice, as they are equivalent. This can be thought of as the periodic version of \cref{lm:derived-torsion-if-homology-torsion}.

\begin{theorem}
    \label{thm:pulling-out-torsion}
    Let $(A, \Psi)$ be an Adams Hopf algebroid and $I\subseteq A$ a finitely generated invariant regular ideal. Then there is an equivalence of stable $\infty$-categories 
    $$[D^{per}(\Psi)]^{I-tors}\simeq D^{per}(\Psi^{I-tors}).$$
\end{theorem}

The proof of this uses the fact that Barr-Beck adjunctions commute with local duality. Proving this here disrupts the flow of the paper, so we deferr it to \cref{app:barr-beck}. 

\begin{proof}
    As $\Comod_\Psi$ is symmetric monoidal we have by \cref{rm:periodic-derived-as-modules} an equivalence
    \[D^{per}(\Psi)\simeq \Mod_{P\1}(D(\Psi)),\]
    coming from the periodicity Barr-Beck adjunction.
    By \cref{thm:modular-bb-torsion} this induces a Barr-Beck adjunction on the torsion subcategories, which gives an equivalence 
    $$[D^{per}(\Psi)]^{I-tors}\simeq \Mod_{\Gamma_I^\Psi (P\1)}(D(\Psi)^{I-tors}).$$
    Since $\Gamma_I^\Psi$ is a smashing colocalization, and $P$ is given by tensoring with $P(\1)$, they do in fact commute. By \cref{lm:derived-torsion-if-homology-torsion} we have $D(\Psi)^{I-tors}\simeq D(\Psi^{I-tors})$, hence the above equivalence can be rewritten as
    $$[D^{per}(\Psi)]^{I-tors}\simeq \Mod_{P(\Gamma_I^\Psi \1)}(D(\Psi^{I-tors})).$$
    Now, also $\Comod_\Psi^{I-tors}$ is symmetric monoidal, so \cref{rm:periodic-derived-as-modules} gives an equivalence 
    \[D^{per}(\Psi^{I-tors})\simeq \Mod_{P(\Gamma_I^\Psi\1)}(D(\Psi^{I-tors})),\]which finishes the proof.
\end{proof}

\section{Exotic algebraic models}
\label{sec:exotic-algebraic-models}

We now have two sets of local duality diagrams, one coming from chromatic homotopy theory, see \cref{ex:local-duality-chromatic}, and one from the homological algebra of Adams Hopf algebroids, see \cref{ex:local-duality-comod}. We can also pass between these duality theories, by using homology theories. In particular, if we let $E=E_n$ be height $n$ Morava $E$-theory at a prime $p$, then we have the $E$-homology functor $E_*\colon \sp\np\longrightarrow \comod\EE$ converting between homotopy theory and algebra. We can, in some sense, say that $E_*$ approximates homotopical information by algebraic information. 

The goal of this section is to set up an abstract framework for studying how good such approximations are. The version we recall below was developed in \cite{patchkoria-pstragowski_2021}, taking inspiration from \cite{franke_96} and \cite{pstragowski_2022}.

\subsection{Adapted homology theories}

Adapted homology theories are particularily well behaved homology theories that have associated Adams type spectral sequences giving computational benefits over other homology theories. 

\begin{definition}
    \label{def:homology-theory}
    Let $\C$ be a presentable symmetric monoidal stable $\infty$-category and $\A$ an abelian category with a local grading $[1]$. A functor $H\colon \C\longrightarrow \A$ is called a {\defn conservative homology theory} if:
    \begin{enumerate}
        \item $H$ is additive
        \item for a cofiber sequence $X\to Y\to Z$ in $\C$, then $HX\to HY\to HZ$ is exact in $\A$
        \item there is a natural isomorphism $H(\Sigma X)\equiv (HX)[1]$ for any $X\in \C$
        \item $H$ reflects isomorphisms. 
    \end{enumerate}
\end{definition}

\begin{remark}
    The first two axioms make $H$ a homological functor, the third makes $H$ into a locally graded functor, i.e., a functor that preserves the local grading, and the last makes it a conservative functor. 
\end{remark}

\begin{example}
    Let $R$ be a ring spectrum. Then the functor $\pi_*\colon \Mod_R\longrightarrow \Mod_{R_*}$ defined as $\pi_* M = [\S, M]_*$ is a conservative homology theory. 
\end{example}

\begin{example}
    Let $R$ be a ring spectrum. The functor $R_*(-)\colon \sp \longrightarrow \Mod_R$, defined as the composition 
    $$\sp\overset{R\otimes (-)}\longrightarrow \Mod_R \overset{\pi_*}\longrightarrow \Mod_{R_*},$$
    is a homology theory. If $R$ is of Adams type, then $R_*(-)$ naturally lands in the subcategory $\Comod_{R_*R}$. If we restrict the domain of $R_*$ to the category of $R$-local spectra, then it is a conservative homology theory. For the rest of the paper we will use $R_*$ to denote the restricted conservative homology theory $R_*\colon \sp_R\to \Comod_{R_*R}$.
\end{example}

\begin{remark}
    Recall that we are really interested in the category $\spK$ of $K_p(n)$-local spectra. The spectrum $K_p(n)$ is a field object in $\sp$, and its homotopy groups $\pi_* K_p(n)$ are graded fields. Hence the homology theory $K_p(n)_*\colon \spK\longrightarrow \Comod_{K_*K}$ is too simple to exhibit the algebraicity properties that we want. There is, however, a version of $E_*$-homology on $\spK$, defined by $E_*^\vee (X) = \pi_*L_{K_p(n)}(E\otimes X)$. This functor is unfortunately not a homology theory, and the associated category of comodules is not abelian. This is the reason for instead using the monochromatic category $\M\np$ and the category of $I_n$-power torsion comodules, as these inherit nicer homological properties we can exploit. 
\end{remark}

\begin{definition}
    \label{def:faithful-lift}
    Let $H\colon \C\longrightarrow \A$ be a homology theory and $J$ an injective object in $\A$. An object $\bar{J}\in \C$ is said to be an {\defn injective lift} of $J$ if it represents the functor 
    $$\Hom_\A(H(-),J)\colon \C^{op}\longrightarrow \A b$$
    in the homotopy category $h \C$, i.e. $\Hom_\A(H(-),J)\cong [-,\bar{J}]$. We call $\bar{J}$ a {\defn faithful lift} if the map $H(\bar{J})\longrightarrow J$ coming from the identity on $\bar{J}$ is an equivalence. 
\end{definition}

\begin{definition}
    \label{def:adapted-homology-theory}
    A homology theory $H\colon \C\longrightarrow \A$ is said to be {\defn adapted} if $\A$ has enough injectives, and for any injective $J \in \A$ there is a faithful lift $\bar{J}\in \C$. 
\end{definition}

\begin{example}
    We again return to our two guiding examples $\pi_*\colon \Mod_R\longrightarrow \Mod_{R_*}$ and $R_*\colon \sp_R\longrightarrow \Comod_{R_*R}$, where $R$ is an Adams type ring spectrum. Both functors are conservative adapted homology theories, with faithful lifts provided by Brown representability. 
\end{example}

\begin{remark}
    The definition of an adapted homology theory $H$ states that for any injective $J\in \A$, there is some object $\bar{J}\in \C$ together with an equivalence $[X,\bar{J}]\simeq \Hom_\A(HX, J).$ Because $\A$ has enough injective objects, we can use these equivalences to approximate homotopy classes of maps by repeatedly mapping into injective envelopes. This gives precisely an associated Adams spectral sequence for the homology theory $H$. In fact, Patchkoria and Pstr{\k a}gowski proved that there is a bijection between adapted homology theories and Adams spectral sequences, see \cite[3.24, 3.25]{patchkoria-pstragowski_2021}. The construction of the Adams spectral sequence associated to an adapted homology theory $H\colon \C\longrightarrow \A$ is given in \cite[2.24]{patchkoria-pstragowski_2021}, or alternatively as a totalization spectral sequence in \cite[2.27]{patchkoria-pstragowski_2021}. 
\end{remark}

In our particular interes $R=E_n$, the associated adapted homology theories $\pi_*$ and $E_*$ are even nicer than a general adapted homology theory. This is because the category of comodules is particularily simple. 

\begin{definition}
    \label{def:cohomological-dimension}
    Let $\A$ be a locally graded abelian category with enough injectives. Then the {\defn cohomological dimension} of $\A$ is the smallest integer $d$ such that $\Ext^{s,t}_\A(-,-) \cong 0$ for all $s>d$. 
\end{definition}

\begin{example}
    \label{ex:cohomological-dimension-comodEE}
    Let $n$ be an integer, $p$ a prime such that $p>n+1$ and $E=E_n$ Morava $E$-theory at height $n$. Then by \cite[2.5]{pstragowski_2021} the category $\Comod\EE$ has cohomological dimension $n^2+n$. 
\end{example}

For certain Adams type ring spectra $R$ we get decompositions of the category $\Comod_{R_*R}$ into periodic families of subcategories. Such decompositions allows for the construction of partial inverses to the associated homology theories. 

\begin{construction}
    \label{const:splitting-of-comodules}
    Let $R$ be an Adams-type ring spectrum such that $\pi_*R$ is concentrated in degrees divisible by some positive number $q+1$, i.e., $\pi_m R = 0$ for all $m\neq 0 \mod q+1$. Any comodule $M$ in the category $\Comod_{R_*R}$ splits uniquely into a direct sum of subcomodules $\bigoplus_{\phi \in \Z/q+1} M_\phi$ such that $M_\phi$ is concentrated in degrees divisible by $\phi$. Such a splitting induces a decomposition of the full subcategory of injective objects 
    $$\Comod_{R_*R}^{inj} \simeq \Comod_{R_*R, 0}^{inj}\times \Comod_{R_*R, 1}^{inj}\times \cdots \times \Comod_{R_*R, q}^{inj}$$ 
    where the category $\Comod_{R_*R, \phi}^{inj}$ denotes the full subcategory spanned by injective comodules concentrated in degrees divisible by $\phi$. 
    
    Let {\defn $h_k \C$} denote the homotopy $k$-category of $\C$, obtained by $k+1$-truncating all the mapping spaces in $\C$.  
    The lift associated with each injective via the Adapted homology theory $R_*$ allows us to construct a partial inverse to $R_*$, called the Bousfield functor $\beta^{inj}$ in \cite{patchkoria-pstragowski_2021}. It is a functor $\beta^{inj}\colon \Comod_{R_*R}^{inj}\longrightarrow h_{q+1} \sp_R^{inj}$, where the latter category is the homotopy $(q+1)$-category of the full subcategory of $\sp_R$ containing all spectra $X$ such that $R_*X$ is injective and $[X,Y]\to \Hom_{R_*R}(R_*X, R_*Y)$ is a bijection for all $Y\in \sp_R$. 
\end{construction}
    
In order to mimic this behavior for a general adapted homology theory, Franke introduced the notion of a splitting of an abelian category. 
    
\begin{definition}[\cite{franke_96}]
    \label{def:splitting-of-abelian-category}
    Let $\A$ be an abelian category with a local grading $[1]$. A {\defn splitting} of $\A$ of order $q+1$ is a collection of Serre subcategories $\A_\phi \subseteq \A$ indexed by $\phi \in \Z/(q+1)$ satisfying
    \begin{enumerate}
        \item $[k]\A_n \subseteq \A_{n+k \mod (q+1)}$ for any $k\in \Z$, and 
        \item the functor $\prod_{\phi}\A_\phi\longrightarrow \A$, defined by $(a_\phi)\mapsto \oplus_\phi a_\phi$, is an equivalence of categories. 
    \end{enumerate}
\end{definition}
    
\begin{example}
    \label{ex:splitting-modules}
    As we saw above in \cref{const:splitting-of-comodules}, the category of comodules over an Adams Hopf algebroid $(R_*, R_*R)$, where $R_*$ is concentrated in degrees divisible by $q+1$, has a splitting of order $q+1$. This, then, also holds for the discrete Hopf algebroid $(R_*, R_*)$, giving the module category $\Mod_{R_*}$ a splitting of order $q+1$ as well. 
\end{example}

\begin{example}
    In the case $R=E(1)$ this has been written out in detail in \cite[Section 4]{barnes-roitzheim_2011}. The Serre subcategories are all copies of the category of $p$-local abelian groups together with Adams operations $\psi^k$ for $k\neq 0$ in $\Z_{(p)}$. The shift leaves the underlying module unchanged, but changes the Adams operation. 
\end{example}
    
\begin{definition}
    \label{not:pure-weight}
    We will say that objects $A\in \A_\phi$ are of {\defn pure weight $\phi$}. 
\end{definition}
    
\begin{remark}
    Just as for $\Comod_{R_*R}$, a splitting of order $q+1$ of a locally graded abelian category $\A$ is enough to define, for any adapted homology theory $H\colon \C\longrightarrow \A$, a partial inverse Bousfield functor $\beta^{inj}$, see \cite[Section 7.2]{patchkoria-pstragowski_2021}. 
\end{remark}


\subsection{Exotic homology theories}

In order to make some statements about exotic equivalences a bit simpler, we introduce the concept of exotic adapted homology theories. Note that this is not the way similar results are phrased in \cite{patchkoria-pstragowski_2021}, but the notation serves as a shorthand for the criteria that they use. 

\begin{definition}
    \label{def:k-exotic-homology-theory}
     Let $H\colon \C\longrightarrow \A$ be a homology theory. We say $H$ is {\defn $k$-exotic} if $H$ is adapted, conservative, $\A$ has finite cohomological dimension $d$ and a splitting of order $q+1$ such that $k=d+1-q>0$. 
\end{definition}
    
The remarkable thing about a $k$-exotic homology theory $H\colon \C\longrightarrow \A$ is that it forces the stable $\infty$-category $\C$ to be approximately algebraic. Intuitively: As the order of the splitting is greater than the cohomological dimension, the $H$-Adams spectral sequence is very sparse and well-behaved. There is a partial inverse of $H$ via the Bousfield functor $\beta\colon \A^{inj}\to h_k \C^{inj}$, which forces a certain subcategory of a categorified deformation of $H$ to be equivalent to both $h_k \C$ and $h_k D^{per}(\A)$. This is the contents of Franke's algebraicity theorem. 
    
\begin{theorem}[{\cite[7.56]{patchkoria-pstragowski_2021}}]
    \label{thm:franke-algebraicity}
    Let $H\colon \C\longrightarrow \A$ be a $k$-exotic homology theory. Then there is an equivalence of homotopy $k$-categories $h_k \C \simeq h_k D^{per}(\A).$
\end{theorem}
    
There are several interesting examples of homology theories satisfying \cref{thm:franke-algebraicity}, see Section 8 in \cite{patchkoria-pstragowski_2021}. We highlight again our two guiding examples but focus specifically on certain Morava $E$-theories. 

\begin{example}[{\cite[8.7]{patchkoria-pstragowski_2021}}]
    \label{ex:chromatic-algebraicity-modules}
    Let $p$ be a prime, $n$ be a non-negative integer, and $E$ a height $n$ Morava $E$-theory concentrated in degrees divisible by $2p-2$, for example Johnson-Wilson theory $E(n)$. If $k=2p-2-n>0$, then the functor $\pi_* \colon \ModE \longrightarrow \modE$ is a $k$-exotic homology theory, giving an equivalence 
    $$h_k \ModE \simeq h_k D^{per}\modE.$$
\end{example}
    
\begin{notation}
    For the following example and the rest of the paper, we follow the notation of \cite{barthel-schlank-stapleton_2020}, \cite{barthel-schlank-stapleton_2021} and \cite{barkan_2023} and denote the category $D^{per}(\comod\EE)$ by {\defn $\Fr\np$}. 
\end{notation}
    
\begin{example}[{\cite[8.13]{patchkoria-pstragowski_2021}}]
    \label{ex:chromatic-algebraicity}
    Let $p$ be a prime, $n$ be a non-negative integer, and $E$ any height $n$ Morava $E$-theory. If $k=2p-2-n^2-n>0$, then the functor $E_* \colon \sp\np\longrightarrow \comod\EE$ is a $k$-exotic homology theory, giving an equivalence 
    \[h_k \sp\np \simeq h_k \Fr\np.\]
\end{example}

\begin{remark}
    As noted in \cite[5.29]{barthel-schlank-stapleton_2020}, this equivalence is strictly exotic for all $n\geq 1$ and primes $p$. In other words, it can never be made into an equivalence of stable $\infty$-categories. In particular, the mapping spectra in $\Fr\np$ are $H\Z$-linear, while the mapping spectra in $\sp\np$ are only $H\Z$-linear for $n=0$. 
\end{remark}
    
\begin{definition}
    Let $H\colon \C\longrightarrow \A$ be a $k$-exotic homology theory. The category $D^{per}(\A)$ is called an {\defn exotic algebraic model} of $\C$ if the equivalence $h_k \C \simeq h_k D^{per}(\A)$ can not be enhanced to an equivalence of $\infty$-categories $\C\simeq D^{per}(\A).$
\end{definition}

\begin{remark}
    The notion of being exotically algebraic is part of a complex hierarchy of algebraicity levels, see \cite{ishak-roitzheim-williamson_2023} for a great exposé. 
\end{remark}
    
\begin{remark}
    The existence of an exotic algebraic model for a stable $\infty$-category $\C$ implies that the category is not rigid. This means, in particular, that there cannot exist a $k$-exotic homology theory with source $\sp$ or $\sp_{(p)}$ as these are all rigid for all primes, see \cite{schwede_07}, \cite{schwede-schipley_02} and \cite{schwede_01}. The same holds for $\sp_{1,2}$, as this is rigid by \cite{roitzheim_07}, and similarily for $\sp_{K_2(1)}$ by \cite{ishak_19}. This shows that being $k$-exotic is quite a strong requirement. 
\end{remark}

\section{Algebraicity for monochromatic categories}

We are now ready to prove our main results. 

\subsection{Monochromatic modules}
\label{ssec:algebraicity-modules}

We start by proving \cref{thm:C}, which we will prove in three steps. We first show it is a conservative adapted homology theory, then that we have finite cohomological dimension and lastly that we have a splitting. For the rest of this section, we assume that $E$ is a version of height $n$ Morava $E$-theory at the prime $p$ that is concentrated in degrees divisible by $2p-2$, for example, $E(n)$ or $E_n^{h\F_p^\times}$.  

The following lemma is the $I_n$-torsion version \cite[3.14]{barthel-frankland_15}, and the proof is similar.

\begin{lemma}
    \label{lm:monochromatic-iff-torsion-modules}
    If $M$ is an $E$-module, then $M\in \Modt$ if and only if $\pi_* M\in \modt$. 
\end{lemma}
\begin{proof}
    Let $X\in \Modt$. By \cite[3.19]{barthel-heard-valenzuela_2018} there is a strongly convergent spectral sequence of $E(n)_*$-modules with signature 
    $$E_2^{s,t} = (H_{I_n}^{-s}\pi_* X)_t \implies \pi_{s+t}M_n X,$$
    where $H_{I_n}^{-s}$ denotes local cohomology. By \cite[2.1.3(ii)]{brodmann-sharp_1998} the $E_2$-page consist of only $I_n$-power torsion modules. As $\modt$ is abelian, it is closed under quotients and subobjects, as as the higher pages are created from the $E_2$-page using quotients and subobjects, they must also consist of only $I_n$-power torsion modules. In particular, the $E_\infty$-page is all $I_n$-power torsion. By Grothendieck's vanishing theorem, see for example \cite[6.1.2]{brodmann-sharp_1998}, $H_{I_n}^s(-)\cong 0$ for $s>n$, hence the abutment of the spectral sequence $\pi_* M_n X$ is a finite fintration of $I_n$-power torsion $E_*$-modules, and is therefore itself an $I_n$-power torsion module. Since $X$ was assumed to be monochromatic, i.e. $X\in \Modt$, we have $\pi_* M_n X\cong \pi_* X$, and thus $\pi_* X\in \modt$. 

    Assume now $X\in \ModE$ such that its homotopy groups are $I_n$-power torsion. Monochromatization gives a map $\phi\colon M_n X\longrightarrow X$, and as $\pi_*M_nX$ is $I_n$-power torsion this map factors on homotopy groups as 
    $$\pi_* M_n X\longrightarrow H^0_{I_n}\pi_* X\longrightarrow \pi_* X,$$
    where the first map is the edge morphism in the above-mentioned spectral sequence. As $\pi_* X$ was assumed to be $I_n$-power torsion we have $\pi_*X\cong H^0_{I_n}\pi_* X$, and $H^s_{I_n}\pi_* X \cong 0$ for $s>0$. Hence the spectral sequence collapses to give the isomorphism $\pi_* M_n X\cong H^0_{I_n}\pi_* X$, which shows that $\pi_* \phi$ is an isomorphism. As $\pi_*$ is conservative $\phi$ was already an isomorphism, hence $X\in \Modt$. 
\end{proof}

\begin{lemma}
    \label{lm:conservative-adapted-torsion-modules}
    Let $p$ be a prime and $n$ a natural number. Then the functor 
    $$\pi_*\colon \Modt\longrightarrow \modt$$
    is a conservative adapted homology theory. 
\end{lemma}
\begin{proof}
    We first note that the functor $\pi_*\colon \ModE\longrightarrow \modE$ is a conservative adapted homology theory. By \cref{lm:monochromatic-iff-torsion-modules} its restriction to $\Modt$ lands in $\modt$, hence autmoatically $\pi_*\colon \Modt\longrightarrow \modt$ is a conservative homology theory. 
    
    Let $J$ be an injective $I_n$-power torsion $E$-module. By the proof of \cite[3.16]{barthel-heard-valenzuela_2020} we can assume $J= T^{E_*}_{I_n} Q$ for an injective $E$-module $Q$. Since $\pi_*$ is adapted on $\ModE$ we can chose a faithful injective lift $\bar{J}$ of $J$ to $\ModE$, and since $\bar{J}$ was assumed to have $I_n$-torsion homotopy groups we know by \cref{lm:monochromatic-iff-torsion-modules} that $\bar{J}\in \Modt$. In particular, we have faithful lifts for any injective in $\modt$, which means that $\pi_*\colon \Modt\longrightarrow \modt$ is adapted. 
\end{proof}

\begin{lemma}
    \label{lm:cohomological-dimension-torsion-modules}
    Let $p$ be a prime and $n$ a natural number. Then the category $\modt$ has cohomological dimension $n$. 
\end{lemma}
\begin{proof}
    Note first that the category $\modE$ has cohomological dimension $n$, and that $\Ext$-groups in $\modt$ are computed in $\modE$. By \cite[2.1.4]{brodmann-sharp_1998}, this implies that the cohomological dimension of $\modt$ can't be greater than $n$, so it remains to prove that it is exactly $n$. We prove this by computing an $\Ext^n_{E_*}$ group that is non-zero.
    
    By \cite[A.2(d)]{hovey-strickland_99} we have $L_0 M \cong \Ext_{E_*}^n(H_{I_n}^n(E_*), M)$ for any $E_*$module $M$. In other words, the derived completion of an $E_*$-module is the $n$'th derived functor of maps from the $I_n$-local cohomology of $E_*$ into $M$. Choosing $M=E_*/I_n$ we get 
    $$L_0 (E_*/I_n)\cong \Ext_{E_*}^n(H^n_{I_n}(E_*), E_*/I_n).$$
    As any bounded $I_n$-torsion $E_*$-module is $I_n$-adically complete we have, as remarked in \cite[1.4]{barthel-heard_16}, that $L_0 (E_*/I_n)\cong E_*/I_n$. The local cohomology of $E_*$ is also $I_n$-torsion, in particular $H_{I_n}^n E_* = E_*/I_n^\infty$. Hence we have 
    $$\Ext_{E_*}^n(E_*/I_n^\infty, E_*/I_n)\cong E_*/I_n \not \cong 0,$$
    showing that there are two $I_n$-power torsion $E_*$-modules with non-trivial $n$'th $\Ext$, which concludes the proof.
\end{proof}

\begin{lemma}
    \label{lm:splitting-torsion-modules}
    Let $p$ be a prime and $n$ a natural number. Then, the category $\modt$ has a splitting of order $2p-2$. 
\end{lemma}
\begin{proof}
    By \cite[8.1]{patchkoria-pstragowski_2021} the category $\modE$ has a splitting of order $2p-2$. We will use this to induce a splitting on $\modt$. In particular, we define the pure weight $\phi$ component of $\modt$, denoted $\Mod_{E_*, \phi}^{I_n-tors}$, to be the essential image of $T_{I_n}^{E_*}\colon \modE\longrightarrow \modt$ restricted to the pure weight $\phi$ component $\Mod_{E_*, \phi}$. We claim that this defines a splitting of order $2p-2$ on $\modt$. For the claim to be true, we need to check the axioms in \cref{def:splitting-of-abelian-category}: $(1)$ that the pure weight components are Serre subcategories, $(2)$ that they are shift-invariant and $(3)$ that they form a decomposition of $\modt$. 

    The first point, $(1)$, follows from the fact that $\Mod_{E_*, \phi}$ is a Serre subcategory, and being $I_n$-power torsion is a property closed under subobjects, quotients, and extensions. Hence also $\Mod_{E_*, \phi}^{I_n-tors}$ is a Serre subcategory. 

    For $(2)$, we note that we have a diagram of adjoint functors 
    \begin{center}
        \begin{tikzcd}
            \modE \arrow[r, "{[1]}", yshift=2] \arrow[d, "T_{I_n}^{E_*}", xshift=2] & \modE \arrow[l, yshift=-2] \arrow[d, "T_{I_n}^{E_*}", xshift=2] \\
            \modt \arrow[r, "{[1]}", yshift=2] \arrow[u, "i", xshift=-2] & \modt \arrow[l, yshift=-2] \arrow[u, "i", xshift=-2]
        \end{tikzcd}
    \end{center}
    which is commutative from bottom left to top right. Here $[1]$ denotes the local grading on $\modt$. We want the diagram to commute from top left to bottom right, which can be obtained by the dual Beck-Chevalley condition. This reduces to checking $[-1]\circ i \simeq i\circ [-1]$, which is true due to the commutativity and the fact that $[1]$ and $[-1]$ are autoequivalences. Hence we have $[1]\circ T_{I_n}^{E_*} \simeq T_{I-n}^{E_*}\circ [1]$. In fact, the diagram is commutative in all possible directions. This means that for any $I_n$-power torsion $E_*$-module $M$ of pure weight $\phi$, we have 
    $$[k]M \cong [k]T_{I_n}^{E_*}M \cong T_{I_n}^{E_*}[k]N \in \Mod_{E_*, \phi + k \mod 2p-2}^{I_n-tors}$$
    as $[k]M\in \Mod_{E_*, \phi+k \mod 2p-2}$. 
    
    For the final point $(3)$, note that any subcategory of a product category is a product of subcategories. Hence, the $\modt$ splits as a product of the pure weight components. In particular, the functor 
    $$\prod_{\phi \in \Z/(2p-2)}\Mod_{E_*, \phi}^{I_n-tors}\longrightarrow \modt$$
    defined by $(M_\phi) \longmapsto \bigoplus_\phi M_\phi$ is an equivalence of categories. 
\end{proof}

\begin{remark}
    This is the part where it was important we chose a version of Morava $E$-theory that is concentrated in degrees divisible by $2p-2$. If we instead chose a $2$-periodic $E$-theory, for example $E_n$, then neither $\modE$ nor $\modt$ would have a splitting of the above degree. 
\end{remark}

We can now summarize the above discussion with the first of our main results. 

\begin{theorem}[\cref{thm:C}]
    \label{thm:main-modules}
    Let $p$ be a prime, $n$ a natural number, and $E$ a version of height $n$ Morava $E$-theory concentrated in degrees divisible by $2p-2$. If $k=2p-2-n>0$, then the functor 
    $$\pi_*\colon \Modt\longrightarrow\modt$$
    is a $k$-exotic homology theory, giving an equivalence 
    $$h_k \Modt \simeq h_k \Dp(\modt).$$
    In particular, monochromatic $E$-modules are exotically algebraic at large primes. 
\end{theorem}
\begin{proof}
    By \cref{lm:cohomological-dimension-torsion-modules} the cohomological dimension of $\modt$ is $n$, and by \cref{lm:splitting-torsion-modules} we have a splitting on $\modt$ of order $2p-2$. Hence, by \cref{lm:conservative-adapted-torsion-modules} the functor $$\pi_*\colon \Modt \longrightarrow \modt$$
    is a $k$-exotic homology theory for $k=2p-2-n>0$, which gives an equivalence 
    $$h_k \Modt \simeq h_k D^{per}(\modt)$$
    by \cref{thm:franke-algebraicity}.
\end{proof}

We can also phrase this dually in terms of $K_p(n)$-local $E$-modules. 

\begin{corollary}
    \label{cor:main-modules-dual}
    Let $p$ be a prime and $n$ a positive integer. Let further $K_p(n)$ be the height $n$ Morava $K$-theory at the prime $p$ and $E$ be a height $n$ Morava $E$-theory at $p$ concentrated in degrees divisible by $2p-2$. If $k=2p-2-n>0$, then we have a $k$-exotic algebraic equivalence 
    $$h_k L_{K_p(n)}\ModE \simeq h_k D^{per}(\modE)^{I_n-comp}.$$ 
    In particular, $K_p(n)$-local $E$-modules are exotically algebraic at large primes. 
\end{corollary}
\begin{proof}
    The equivalence is constructed from the equivalences obtained from \cref{rm:local-duality-modules}, \cref{thm:main-modules}, \cref{thm:pulling-out-torsion} and \cref{const:periodic-derived-local-duality}. In particular, we have
    \begin{align*}
        h_k \Modc
        &\hspace{2pt}\overset{\ref{rm:local-duality-modules}}\simeq 
        h_k \hspace{2pt} \Modt \\
        &\hspace{2pt}\overset{\ref{thm:main-modules}}\simeq 
        \hspace{2pt} h_k D^{per}(\modt) \\
        &\overset{\ref{thm:pulling-out-torsion}}\simeq
        h_k D^{per}(\modE)^{I_n-tors} \\
        &\overset{\ref{const:periodic-derived-local-duality}}\simeq 
        h_k D^{per}(\modE)^{I_n-comp},
    \end{align*}
    where we have used that an equivalence of $\infty$-categories induces an equivalence on homotopy $k$-categories.
\end{proof}

Now, let $HE_*$ be the Eilenberg-MacLane spectrum of $E_*$. By Schwede's derived morita theory, see \cite[7.1.1.16]{Lurie_HA}, there is a symmetric monoidal equivalence $D(E_*)\simeq \Mod_{HE_*}$, and we can form a local duality diagram for $\Mod_{HE_*}$ corresponding to \cref{ex:local-duality-comod} for the discrete Hopf algebroid $(E_*, E_*)$. By arguments similar to \cref{lm:monochromatic-iff-torsion-modules} and \cref{lm:conservative-adapted-torsion-modules} one can show that the homotopy groups functor $\pi_*\colon \Mod_{HE_*}\longrightarrow \ModE$ restricts to a conservative adapted homology theory 
$$\pi_* \Mod_{HE_*}^{I_n-tors}\longrightarrow \modt.$$
In the same range as \cref{thm:main-modules} this is then automatically also $k$-exotic. We can then combine the algebraicity for $\Modt$ and $\Mod_{HE_*}$ to get the following statement. 

\begin{corollary}
    Let $k=2p-2-n>0$. Then, there is an exotic equivalence 
    $$h_k \Modt \simeq h_k \Mod_{HE_*}^{I_n-tors}.$$
\end{corollary}

\subsection{Monochromatic spectra}
\label{ssec:algebraicity-spectra}

Having proven that monochromatic $E$-modules are algebraic at large primes, we now turn to the larger category of all monochromatic spectra $\M\np$ with the same goal. The strategy is exactly the same as in \cref{ssec:algebraicity-modules}: we first prove that the conservative adapted homology theory $E_*\colon \sp\np\longrightarrow \comod\EE$ restricts to a conservative adapted homology theory on $\M\np$, before proving that $\Comod\EE^{I_n-tors}$ has a splitting and finite cohomological dimension. This will prove \cref{thm:B}, which we then convert into a proof of \cref{thm:A}, as in \cref{cor:main-modules-dual}.

\begin{lemma}
    \label{lm:monochromatic-iff-torsion-comodules}
    If $X$ is a $E$-local spectrum, then $X\in \M\np$ if and only if $E_*X\in \comodt$. 
\end{lemma}
\begin{proof}
    Assume first that $X\in \M\np$. We have $E\otimes X\in \ModE^{I_n-tors}$ as
    $$E\otimes X\simeq E\otimes M_n X\simeq M_n E\otimes X,$$
    where the last equivalence follows from $M_n$ being smashing. In particular, the restricted functor $E_*\colon \M\np\longrightarrow \comod\EE$ factors through $\ModE^{I_n-tors}$. By \cref{lm:monochromatic-iff-torsion-modules} and \cref{rm:torsion-iff-underlying-is-torsion} this means that $E_*X$ is an $I_n$-power torsion $E_*E$-comodule. 

    For the converse, assume that we have $X\in \sp\np$ such that $E_*X\in \comodt$. Using the monochromatization functor we obtain a comparison map $M_n X\longrightarrow X$, which induces a map on $E$-modules $E\otimes M_n X\longrightarrow E\otimes X$. This map is an isomorphism on homotopy groups, as $E_*X$ was assumed to be $I_n$-power torsion. As $E_*$ is conservative on $\sp\np$, the original comparison map $M_n X\longrightarrow X$ was an isomorphism, meaning that $X\in \M\np$. 
\end{proof}

\begin{lemma}
    \label{lm:conservative-adapted-torsion-comodules}
    Let $p$ be a prime and $n$ a natural number. Then, the functor 
    $$E_*\colon \M\np\longrightarrow \comodt$$
    is a conservative adapted homology theory. 
\end{lemma}
\begin{proof}
    First note that the image of the functor $E_*\colon \sp\np\longrightarrow \comod\EE$ restricted to $\M\np$ is contained in $\comodt$ by \cref{lm:monochromatic-iff-torsion-comodules}. The functor $E_*\colon \M\np\longrightarrow \comodt$ is then automatically a conservative homology theory. The category $\comodt$ has enough injectives as it is Grothendieck by \cref{rm:torsion-comodules-grothendieck-monoidal}. Hence, it only remains to prove that we have faithful lifts for all injective objects. 

    Let $J$ be an injective in $\comodt$. Following \cite[3.16]{barthel-heard-valenzuela_2020} we can assume that $J = T^{E_*E}_{I_n}(E_*E\otimes_{E_*}Q)$ for some injective $E_*$-module $Q$. Since being torsion is a property of the underlying module, and the forgetful functor $\epsilon_*$ is conservative, we have an isomorphism $T^{E_*E}_{I_n}(E_*E\otimes_{E_*}Q)\cong E_*E\otimes_{E_*}T_{I_n}^{E_*}Q.$ By the proof of \cref{lm:conservative-adapted-torsion-modules}, $T_{I_n}^{E_*}Q$ is injective in $\Mod_{E_*}$, hence $J$ is also injective in $\Comod\EE$.  

    Now, $E_*$ has faithful injective lifts from $\comod\EE$ to $\sp\np$, hence there is a lift $\bar{J}$ such that $[X,\bar{J}]\simeq \Hom_{E_*E}(E_*X, J)$ and $E_*\bar{J}\simeq J$. By \cref{lm:monochromatic-iff-torsion-comodules} $\bar{J}\in \M\np$ as $J$ was assumed to be $I_n$-power torsion, hence we have found our faithful injective lift. 
\end{proof}

\begin{lemma}
    \label{lm:cohomological-dimension-torsion-comodules}
    Let $p$ be a prime and $n$ a natrural number. If $p-1\nmid n$, then the category $\comodt$ has cohomological dimension $n^2+n$. 
\end{lemma}
\begin{proof}
    The proof follows \cite[2.5]{pstragowski_2021} closely, which is itself a modern reformulation of \cite[3.4.3.9]{franke_96}. As in \cref{lm:cohomological-dimension-torsion-modules} we note that also $\Ext$-groups in $\comodt$ are computed in $\comod\EE$. We start by defining {\defn good targets} to be $I_n$-power torsion comodules $N$ such that $\Ext_{E_*E}^{s,t}(E_*/I_n, N)=0$ for all $s>n^2+n$ and {\defn good sources} to be $I_n$-power torsion comodules $M$ such that $\Ext_{E_*E}^{s,t}(M,N)=0$ for all $s>n^2+n$ and $I_n$-torsion comodules $N$. 

    By the Landweber filtration theorem, see for example \cite[5.7]{hovey-strickland_99}, we know that any finitely presented comodule $M$ has a finite filtration 
    $$0=M_0 \subseteq M_1 \subseteq \cdots \subseteq M_{s-1}\subseteq M_s=M,$$
    where $M_r/M_{r-1} \cong E_*/I_{j_r}[t_r]$ and $j_r\leq n$. When $M$ is $I_n$-power torsion we get $j_r=n$ for all $r$, as noted in \cite[4.3]{hovey-strickland_99}. By Morava's vanishing theorem, see for example \cite[6.2.10]{ravenel_86}, we have $\Ext_{E_*E}^{s,t}(E_*, E_*/I_n) = 0$ for all $s>n^2$. By inductively using the short exact sequences
    \[0\longrightarrow E_* \overset{v_i}\longrightarrow E_* \longrightarrow E_*/v_i\longrightarrow 0\]
    for $0\leq i\leq n$ and the induced long exact sequence in $\Ext$-groups, we get that 
    \[\Ext_{E_*E}^{s,t}(E_*/I_n, E_*/I_n) = 0\] 
    for $s>n^2+n$. By using the Landweber filtration, this implies that any finitely presented $I_n$-power torsion comodule is a good target. By \cref{lm:torsion-comodules-generated-by-compacts} any $I_n$-power torsion comodule is a filtered colimit of finitely presented ones, and as $\Ext_{E_*E}^{s,t}(E_*/I_n, -)$ commutes with colimits this implies that any $I_n$-power torsion comodule is a good target. 

    Note that the above argument also proves that $E_*/I_n$ is a good source, which by the Landweber filtration argument implies that any finitely presented $I_n$-torsion comodule is a good source. By \cref{lm:torsion-comodules-generated-by-compacts}, the category $\comodt$ is generated under filtered colimits by finitely presented ones. Hence, we can apply \cite[2.4]{pstragowski_2021} to any injective resolution 
    $$0\longrightarrow M \longrightarrow J_0 \longrightarrow J_1\longrightarrow \cdots$$
    to get that the map $J_{n^2+n}\longrightarrow \ima(J_{n^2+n}\longrightarrow J_{n^2+n+1})$ is a split surjection, and that the object $\ima(J_{n^2+n}\longrightarrow J_{n^2+n+1})$ is injective. Hence, any injective resolution can be modified to have length $n^2+n$, which concludes the proof. 
\end{proof}

\begin{remark}
    In a previous version of this paper, we claimed that the cohomological dimension was $n^2$. We want to thank Piotr Pstr\a{}gowski for pointing out the gap in the proof. This means that $\Comod\EE^{I_n-tors}$ has the same cohomological dimension as the non-torsion category $\comod\EE$, as seen in \cref{ex:cohomological-dimension-comodEE}. However, we do obtain something slightly stronger, as our result holds for all $p-1\nmid n$, while the analogue in $\Comod\EE$ only holds when $p-1>n$. In fact, $\Comod\EE$ does not have finite cohomological dimension when $p-1\leq n$, as noted in \cite[2.6]{pstragowski_2021}. 
\end{remark}

\begin{lemma}
    \label{lm:splitting-torsion-comodules}
    Let $p$ be a prime, $n$ a natural number, and $E$ any height $n$ Morava $E$-theory. Then, the category $\comodt$ has a splitting of order $2p-2$. 
\end{lemma}
\begin{proof}
    All of the height $n$ Morava $E$-theories have equivalent categories of comodules by \cite[4.2]{hovey-strickland_2005a}. Hence we can chose a version concentrated in degrees divisible by $2p-2$, giving $\comod\EE$ a splitting of order $2p-2$, see \cite[8.13]{patchkoria-pstragowski_2021}. The proof of the induced splitting on the $I_n$-torsion category is then identical to \cref{lm:splitting-torsion-modules}. 
\end{proof}

We can now summarize the above results with our second main result, which is the monochromatic analogue of \cref{ex:chromatic-algebraicity}. 

\begin{theorem}[\cref{thm:B}]
    \label{thm:main-spectra}
    Let $p$ be a prime, $n$ a natural number, and $E$ any height $n$ Morava $E$-theory. If $k=2p-2-n^2-n>0$, then the restricted functor $E_*\colon \M\np\longrightarrow \comodt$ is $k$-exotic. In particular, there is an equivalence 
    \begin{equation*}
        h_k \M\np \simeq h_kD^{per}(E_*E^{I_n-tors}),
    \end{equation*}
    meaning that monochromatic homotopy theory is exotically algebraic at large primes. 
\end{theorem}
\begin{proof}
    By \cref{lm:cohomological-dimension-torsion-comodules}, the cohomological dimension of $\comodt$ is $n^2+n$ and by \cref{lm:splitting-torsion-comodules} we have a splitting of order $2p-2$. The restricted functor $E_*$ is then by \cref{lm:conservative-adapted-torsion-comodules} $k$-exotic whenever $k=2p-2-n^2-n>0$, which by \cref{thm:franke-algebraicity} finishes the proof. 
\end{proof}

\begin{remark}
    \label{rm:monochromatic-corresponds-to-torsion}
    By \cref{thm:pulling-out-torsion} there is an equivalence $D^{per}(E_*E^{I_n-tors})\simeq \Fr\np^{I_n-tors}$ and by \cref{ex:local-duality-chromatic} there is an equivalence $\M\np\simeq \sp\np^{I_n-tors}$. This means that we can write the equivalence in \cref{thm:main-spectra} as 
    $$h_k\sp\np^{I_n-tors}\simeq h_k\Fr\np^{I_n-tors}$$
    for $k=2p-2-n^2-n>0$. This is more in line with thinking about \cref{thm:main-spectra} as ``coming from'' the chromatic algebraicity of \cref{ex:chromatic-algebraicity} on localizing ideals. This formulation is perhaps also easier to connect to the limiting case $p\to \infty$ as described using ultraproducts in \cite{barthel-schlank-stapleton_2021}, which can be stated informally as 
    $$\lim_{p\to \infty} \sp\np^{I_n-tors}\simeq \lim_{p\to \infty} \Fr\np^{I_n-tors}.$$
\end{remark}

Via \cref{thm:local-duality} we can now obtain the associated exotic algebraicity statement for the category of $K_p(n)$-local spectra.

\begin{theorem}[\cref{thm:A}]
    \label{thm:main-spectra-dual}
    Let $p$ be a prime and $n$ a natural number. Let further $K_p(n)$ be height $n$ Morava $K$-theory at the prime $p$ and $E$ be any height $n$ Morava $E$-theory at $p$. If $k=2p-2-n^2>0$, then we have a $k$-exotic algebraic equivalence 
    $$h_k \spK \simeq h_k \Fr\np^{I_n-comp}.$$ 
\end{theorem}
\begin{proof}
    As we did in \cref{cor:main-modules-dual}, we construct the equivalence from a sequence of equivalences coming from \cref{thm:local-duality} and \cref{thm:main-spectra}. More precisely we use equivalences coming from \cref{ex:local-duality-chromatic}, \cref{thm:main-spectra}, \cref{thm:pulling-out-torsion} and \cref{const:periodic-derived-local-duality}, which give
    \begin{align*}
        h_k \spK
        &\overset{\hspace{2pt}\ref{ex:local-duality-chromatic}\hspace{2pt}}
        \simeq 
        h_k \M\np \\
        &\overset{\ref{thm:main-spectra}}
        \simeq 
        h_k D^{per}(\comodt) \\
        &\overset{\ref{thm:pulling-out-torsion}}
        \simeq 
        h_k \Fr\np^{I_n-tors} \\
        &\overset{\ref{const:periodic-derived-local-duality}}
        \simeq 
        h_k \Fr\np^{I_n-comp},
    \end{align*}
    where we again have used that an equivalence of $\infty$-categories induces an equivalence on homotopy $k$-categories.
\end{proof}

\begin{remark}
    As in \cref{rm:monochromatic-corresponds-to-torsion} we can phrase \cref{thm:main-spectra-dual} as $h_k\sp\np^{I_n-comp}\simeq h_k\Fr\np^{I_n-comp}$. 
\end{remark}


\subsection*{Some remarks on future work}

The reason why \cref{thm:franke-algebraicity} works so well, is that there is a deformation of stable $\infty$-categories lurking behind the scenes. One does not need this in order to apply the theorem, but it is there regardless. In the case of a Morava $E$-theory $E=E_n$, the deformation associated with the adapted homology theory $E_*\colon \sp\np \longrightarrow \comod\EE$ is equivalent to the category of hypercomplete $E$-based synthetic spectra, $\hsynE$, introduced in \cite{pstragowski_2022}. Our restricted homology theory $E_*\colon \M_{n,p}\longrightarrow \comodt$ should then be associated to a deformation $\hsynE^{I_n-tors}$ coming from a local duality theory for $\hsynE$, in the sense that there is a diagram of stable $\infty$-categories 
\begin{center}
    \begin{tikzcd}
        \M_{n,p}\simeq\sp\np^{I_n-tors} & \hsynE^{I_n-tors} \arrow[l, "\tau^{-1}"'] \arrow[r, "\tau\sim 0"] & \Fr\np^{I_n-tors}.
    \end{tikzcd}    
\end{center}
Since $E_*$ is adapted on $\M\np$, we abstractly know that there is a deformation $D^\omega(\M\np)$ arising out of the work of Patchkoria-Pstr{\k a}gowski in \cite{patchkoria-pstragowski_2021}, called the perfect derived category. This should give an equivalent ``internal'' approach to $I_n$-torsion synthetic spectra, much akin to how we have equivalences $\M\np\simeq \sp\np^{I_n-tors}$ and $D(E_*E)^{I_n-tors}\simeq D(E_*E^{I_n-tors})$. 

In \cite{barkan_2023}, Barkan provides a monoidal version of \cref{thm:franke-algebraicity} by using filtered spectra. His deformation $\mathscr{E}\np$ is equivalent to $\hsynE$, which hints that there should be a monoidal version of \cref{thm:main-spectra} as well. We originally intended to incorporate such a result into this paper but decided against it in order to keep it free from deformation theory. We do, however, state the conjectured monoidal result, which we hope to pursue in future work.

\begin{conjecture}
    Let $p$ be a prime and $n$ a natural number. If $k$ is a positive natural number such that $2p-2>n^2+(k+3)n+k-1$, then we have an equivalence $h_k \M\np\simeq h_k \Fr\np^{I_n-tors}$ of symmetric monoidal stable $\infty$-categorires. 
\end{conjecture}

As \cref{thm:local-duality} is monoidal, this would give a similar statement for the $K_p(n)$-local category, i.e. a symmetric monoidal equivalence $h_k \spK\simeq h_k \Fr\np^{I_n-comp}.$

Since $E$-based synthetic spectra are categorifications of the $E$-Adams spectral sequence, one should expect the above-mentioned local duality for $\hsynE$ to give a category $\hsynE^{I_n-comp}$, which categorifies the $K_p(n)$-local $E$-Adams spectral sequence. We plan to study such categorifications of the $K_p(n)$-local $E$-Adams spectral sequence in future work joint with Marius Nielsen. 
\appendix

\section{Barr-Beck for localizing ideals}
\label{app:barr-beck}

In this appendix we prove that the monoidal Barr-Beck theorem---a monoidal version of Lurie's $\infty$-categorical version of the classical Barr-Beck monadicity theorem, see \cite[Section 4.7]{Lurie_HA}---interacts nicely with local duality.

\begin{theorem}[{\cite[5.29]{mathew-naumann-noel_2017}}]
    \label{thm:modular-bb}
    Let $\C, \D \in \Alg(\Pr)$ and $(F\dashv G)\colon \C\longrightarrow \D$ be a monoidal adjunction. If in addition 
    \begin{enumerate}
        \item $G$ is conservative, 
        \item $G$ preserves colimits, and
        \item the projection formula holds,
    \end{enumerate}
    then $(F,G)$ is a monoidally monadic adjunction and the monad $GF$ is equivalent to the monad $G(\1_\D)\otimes (-)$. In particular this gives a symmetric monoidal equivalence $\D\simeq \Mod_{G(\1_\D)}(\C).$
\end{theorem}
\begin{proof}
    By \cite[4.7.0.3]{Lurie_HA} the adjunction is monadic by the first two criteria, giving an equivalence $\D\simeq \Mod_{GF}(\C)$. The projection formula applied to the unit $\1_\D$ gives an equivalence of monads $GF\simeq G(\1_\D)\otimes \C$. 
\end{proof}

\begin{definition}
    When the three criteria above hold for a given monoidal adjunction $(F\dashv G)$, we will say that the adjunction satisfies the monoidal Barr-Beck criteria or that it is a {\defn monoidal Barr-Beck adjunction}. We will sometimes omit the prefix monoidal when it is clear from context. 
\end{definition}

Let $(\C, \K)$ be a local duality context. We wish to prove that the associated local duality diagram is compatible with \cref{thm:modular-bb}. By modifying \cite[3.7]{behrens-shaw_2020} slightly, we know that any Barr-Beck adjunction induces a Barr-Beck adjunction on $\K$-local and $\K$-complete objects. Hence, it remains only to prove a similar statement for the $\K$-torsion objects. 

\begin{definition}
    Let $(\C, \K)$ and $(\D, \L)$ be local duality contexts. A {\defn map of local duality contexts} is a symmetric monoidal colimit-preserving functor $F\colon \C\longrightarrow \D$ such that $F(\K)\subseteq \L$. If, in addition $\Loc_\D^\otimes(F(\K))\simeq \Loc_\D^\otimes(\L)$, then we say $F$ is a {\defn strict} map of local duality contexts. A monoidal adjunction $(F\dashv G)\colon \C\longrightarrow \D$ such that $F$ is a strict map of local duality contexts is called a {\defn local duality adjunction}, sometimes denoted 
    $$(F\dashv G)\colon (\C, \K)\longrightarrow (\D, \L).$$
\end{definition}

Given a local duality context and an appropriate functor, one can always extend the functor to a strict map of local duality context in the following way. 

\begin{construction}
    Let $(\C, \K)$ be a local duality context, $\D \in \Alg(\Pr)$ and $F\colon \C\longrightarrow \D$ be a symmetric monoidal colimit-preserving functor. The image of $\K$ under $F$ generates a localizing ideal $\Loc^\otimes_\D(F(\K))$ in $\D$, which makes $F$ a map of local duality contexts. We call this the local duality context on $\D$ induced by $\C$ via $F$. 
\end{construction}

The following lemma is essentially the ``non-geometric'' version of \cite[5.11]{balmer-sanders_2017}. The proof is also similar, but as we have phrased it in a different and slightly more general language, we present a full proof. 

\begin{lemma}
    \label{lm:induced-torsion-adjunction}
    Let $(F\dashv G)\colon (\C, \K)\longrightarrow (\D, \L)$ be a local duality adjunction. Then, the adjunction induces a monoidal adjunction on localizing ideals
    \begin{center}
    \begin{tikzcd}
        \Loc_\C^\otimes(\K) \arrow[rr, "F'", yshift=2] &  & \Loc_\D^\otimes(\L) \arrow[ll, "G'", yshift=-2].
    \end{tikzcd}
    \end{center}
\end{lemma}
\begin{proof}
    From \cref{rm:monoidal-structure-in-local-duality} we know that the symmetric monoidal structures on $\Loc^\otimes_\C(\K)$ and $\Loc^\otimes_\D(\L)$ is simply the symmetric monoidal structures on $\C$ and $\D$, restricted to the full subcategories. 
    
    Since $F$ is a map of local duality contexts, we have an inclusion $F(\K)\subseteq \L$, which gives inclusions  
    $$F(\Loc^\otimes_{\C}(\K))\subseteq \Loc^\otimes_\D(F(\K))\subseteq \Loc^\otimes_\D(\L),$$
    meaning that the functor $F$ restricts to the torsion objects. In particular we have for any object $X\in \C^{\K-tors}$ an equivalence $\Gamma_\L F(X) \simeq F(X)$. We let $F'=F_{|\Loc^\otimes_\C(\K)}$ and define $U'$ to be the composition 
    $$Loc^{\otimes}_\D(\L)\overset{i_{\L-tors}}\longrightarrow \D\overset{U}\longrightarrow \C\overset{\Gamma_\K}\longrightarrow \Loc_\C^{\otimes}(\K),$$
    which is an adjoint to $F'$. We need to show that $F$ is a symmetric monoidal functor, but, as the inclusions $i_{\K-loc}$ and $i_{\L-loc}$ are non-unitally monoidal all that remains to be proven is that $F'$ sends the monoidal unit $\Gamma_\K \1_\C$ to the monoidal unit $\Gamma_\L \1_\D$. 

    The localizing ideals $\Loc_\C^\otimes(\K)$ and $\Loc_\D^\otimes(\L)$ are equivalent to the localizing ideals generated by the respective units, i.e. 
    $$\Loc_\C^\otimes(\K)\simeq \Loc_\C^\otimes(\Gamma_\K\1_\C) \quad \text{and}\quad \Loc_\D^\otimes(\L)\simeq \Loc_\D^\otimes(\Gamma_\L\1_\D).$$
    Since $(F\dashv U)$ is a local duality adjunction we also know that $\Loc^\otimes_\D(F(\K))\simeq \Loc_\D^\otimes(\L)$, which also means $\Loc_\D(F(\Gamma_\K \1_\C)) \simeq \Loc_\D^\otimes(\L)$.
    Let $\mathcal{G}$ be the full subcategory of $\Loc_\D^\otimes(\L)$ where $F(\Gamma_\K\1_\C)$ acts as a unit, in other words objects $M\in \Loc_\D^\otimes(\L)$ such that $F(\Gamma_\K\1_\C)\otimes_{\D} M\simeq M$. In particular, $F(\Gamma_\K\1_\C)$ is in $\mathcal{G}$. The category $\mathcal{G}$ is closed under retracts, suspension, and colimits, as well as tensoring with objects in $\D$, as we have 
    $$F(\Gamma_\K\1_\C) \otimes_\D (M\otimes_\D D)\simeq (F(\Gamma_\K\1_\C)\otimes_\D M) \otimes_\D D \simeq M\otimes_\D D$$
    for any $M\in \mathcal{G}$ and $D\in \D$. Hence, it is a localizing tensor ideal of $\D$, with a symmetric monoidal structure where the unit is $F(\Gamma_\K\1_\C)$. In particular, $\mathcal{G}\simeq \Loc_\D^\otimes(F(\Gamma_\K\1_\C))$, which we already know is equivalent to $\Loc_\D^\otimes(\L).$
    
    Since the ideals are equivalent, and the unit is unique, we must have $F(\Gamma_\K\1_\C)\simeq \Gamma_\L \1_D$, which finishes the proof. 
\end{proof}

The key feature for us is that such an induced adjunction inherits the property of being a Barr-Beck adjunction, i.e., that the right adjoint is conservative, preserves colimits, and has a projection formula. An analogous, but not equivalent, statement was proven in \cite[4.5]{behrens-shaw_2020}. Another related, but not equivalent statement, is Greenlees and Shipleys Cellularization principle, see \cite{greenlees-shipley_2013}. 

\begin{theorem}
    \label{thm:modular-bb-torsion}
    Let $(F\dashv G)\colon (\C, \K)\longrightarrow (\D, \L)$ be a local duality adjunction. If $(F\dashv G)$ satisfies the Barr-Beck criteria, then the induced monoidal adjunction on localizing ideals
    \begin{center}
        \begin{tikzcd}
            \Loc_\C^\otimes(\K) \arrow[rr, "F'", yshift=2] &  & \Loc_\D^\otimes(\L) \arrow[ll, "G'", yshift=-2]
        \end{tikzcd}
    \end{center}
    constructed in \cref{lm:induced-torsion-adjunction}, also satisfies the Barr-Beck criteria. 
\end{theorem}
\begin{proof}
    We need to prove that $G'$ is conservative and colimit-preserving and that the projection formula holds. The first two will both follow from the following computation, showing that also $G'$ is just the restriction of $G$ to $\Loc_\D^\otimes(\L)$. 

    Let $X\in \Loc_\D^\otimes(\L)$. By definition we have $G'(X) = \Gamma_\K G(X)$, where we have omitted the inclusions from the notation for simplicity. Since $\Gamma_\K$ is smashing and $(F\dashv G)$ by assumption has a projection formula we have 
    $$\Gamma_\K G(X)\simeq G(X)\otimes_\C \Gamma_\K\1_\C \simeq G(X\otimes_\D F(\Gamma_\K\1_\C)).$$
    By \cref{lm:induced-torsion-adjunction} $F'$ is symmetric monoidal, hence $F(\Gamma_\K\1_\C)\simeq \Gamma_\L \1_\D$, which acts on $X$ as the monoidal unit. Thus, we can summarize with
    $$G'(X)\simeq G(X\otimes_\D F(\Gamma_\K\1_\C))\simeq G(X\otimes_\D \Gamma_\L\1_\D)\simeq G(X),$$
    which shows that also $G'$ is the restriction of $G$. 

    Now, as $U$ is both conservative and preserves colimits, and colimits in the localizing ideals are computed in $\C$ and $\D$ respectively, then also $U'$ is conservative and colimit-preserving. The projection formula for $(F'\dashv U')$ also automatically follows from the projection formula for $(F\dashv U)$.  
\end{proof}

\printbibliography{}

@misc{Lurie_HA,
  author         = {Jacob Lurie},
  title          = {Higher algebra},
  year           = {2017},
  note         = {Available at \href{https://www.math.ias.edu/~lurie/papers/HA.pdf}{the authors website}.},
}

@article{ishak-roitzheim-williamson_2023,
	title = {Levels of algebraicity in stable homotopy theories},
	volume = {108},
	issn = {1469-7750},
	doi = {10.1112/jlms.12752},
	number = {2},
	journal = {Journal of the London Mathematical Society},
	author = {Ishak, Jocelyne and Roitzheim, Constanze and Williamson, Jordan},
	year = {2023},
	pages = {545--577},
}

@incollection{barthel-beaudry_19,
	title = {Chromatic structures in stable homotopy theory},
	isbn = {978-1-351-25162-4},
	booktitle = {Handbook of {Homotopy} {Theory}},
	publisher = {Chapman and Hall/CRC},
	author = {Tobias Barthel and Agnès Beaudry},
	year = {2019},
}

@misc{patchkoria-pstragowski_2021,
      title={Adams spectral sequences and {F}ranke's algebraicity conjecture}, 
      author={Irakli Patchkoria and Piotr Pstr\a{}gowski},
      year={2021},
      eprint={2110.03669},
      archivePrefix={arXiv},
      primaryClass={math.AT}
}

@article{pstragowski_2021,
	title = {Chromatic homotopy theory is algebraic when $p>n^2+n+1$},
	volume = {391},
	issn = {0001-8708},
	url = {https://www.sciencedirect.com/science/article/pii/S0001870821003972},
	doi = {10.1016/j.aim.2021.107958},
	journal = {Advances in Mathematics},
	author = {Pstr\a{}gowski, Piotr},
	year = {2021},
	pages = {107958},
}

@article{pstragowski_2022,
	title = {Synthetic spectra and the cellular motivic category},
	volume = {232},
	issn = {1432-1297},
	url = {https://doi.org/10.1007/s00222-022-01173-2},
	doi = {10.1007/s00222-022-01173-2},
	number = {2},
	journal = {Inventiones Mathematicae},
	author = {Pstr\a{}gowski, Piotr},
	year = {2023},
	pages = {553--681},
}

@misc{barkan_2023,
      title={Chromatic Homotopy is Monoidally Algebraic at Large Primes}, 
      author={Shaul Barkan},
      year={2023},
      eprint={2304.14457},
      archivePrefix={arXiv},
      primaryClass={math.AT}
}

@article{barthel-heard_16,
	title = {The ${E}_2$-term of the ${K}(n)$-local ${E}_n$-{Adams} spectral sequence},
	volume = {206},
	issn = {0166-8641},
	url = {https://www.sciencedirect.com/science/article/pii/S0166864116300050},
	doi = {10.1016/j.topol.2016.03.024},
	journal = {Topology and its Applications},
	author = {Tobias Barthel and Drew Heard},
	year = {2016},
	pages = {190--214},
}

@article{barthel-heard-valenzuela_2018,
	title = {Local duality in algebra and topology},
	volume = {335},
	issn = {0001-8708},
	url = {https://www.sciencedirect.com/science/article/pii/S0001870818302652},
	doi = {10.1016/j.aim.2018.07.017},
	journal = {Advances in Mathematics},
	author = {Tobias Barthel and Drew Heard and Gabriel Valenzuela},
	year = {2018},
	pages = {563--663},
}

@article{barthel-heard-valenzuela_2020,
	title = {Derived completion for comodules},
	volume = {161},
	issn = {0025-2611, 1432-1785},
	url = {http://link.springer.com/10.1007/s00229-018-1094-0},
	doi = {10.1007/s00229-018-1094-0},
	number = {3-4},
	journal = {manuscripta mathematica},
	author = {Tobias Barthel and Drew Heard and Gabriel Valenzuela},
	year = {2020},
	pages = {409--438},
}

@article{barthel-frankland_15,
	title = {Completed power operations for {Morava} $E$-theory},
	volume = {15},
	issn = {1472-2747, 1472-2739},
	url = {https://projecteuclid.org/journals/algebraic-and-geometric-topology/volume-15/issue-4/Completed-power-operations-for-Morava-Etheory/10.2140/agt.2015.15.2065.full},
	doi = {10.2140/agt.2015.15.2065},
	number = {4},
	journal = {Algebraic \& Geometric Topology},
	author = {Tobias Barthel and Martin Frankland},
	year = {2015},
	pages = {2065--2131},
}

@article{balmer-sanders_2017,
	title = {The spectrum of the equivariant stable homotopy category of a finite group},
	volume = {208},
	issn = {1432-1297},
	url = {https://doi.org/10.1007/s00222-016-0691-3},
	doi = {10.1007/s00222-016-0691-3},
	number = {1},
	urldate = {2024-02-22},
	journal = {Inventiones Mathematicae},
	author = {Balmer, Paul and Sanders, Beren},
	year = {2017},
	pages = {283--326},
}

@article{barnes-roitzheim_2011,
	title = {Monoidality of {F}ranke's exotic model},
  volume = {228},
	issn = {0001-8708},
	url = {https://www.sciencedirect.com/science/article/pii/S0001870811002908},
  number = {6},
	journal = {Advances in mathematics},
	author = {David Barnes and Constanze Roitzheim},
  year = {2011},
  pages = {3223--3248},
}

@article{mathew-naumann-noel_2017,
	title = {Nilpotence and descent in equivariant stable homotopy theory},
	volume = {305},
	issn = {0001-8708},
	url = {https://www.sciencedirect.com/science/article/pii/S0001870815300062},
	doi = {10.1016/j.aim.2016.09.027},
	journal = {Advances in Mathematics},
	author = {Mathew, Akhil and Naumann, Niko and Noel, Justin},
	year = {2017},
	pages = {994--1084},
}

@article{behrens-shaw_2020,
	title = {${C}_2$-equivariant stable homotopy from real motivic stable homotopy},
	volume = {5},
	issn = {2379-1691},
	url = {https://msp.org/akt/2020/5-3/p03.xhtml},
	doi = {10.2140/akt.2020.5.411},
	number = {3},
	journal = {Annals of K-Theory},
	author = {Behrens, Mark and Shah, Jay},
	year = {2020},
	pages = {411--464},
}

@book{hovey-strickland_99,
	series = {Memoirs of the {American} {Mathematical} {Society}},
	title = {Morava $K$-theories and localisation},
	volume = {139},
	note = {ISBN: 978-0-8218-1079-8},
	url = {https://www.ams.org/memo/0666},
	publisher = {American Mathematical Society},
	author = {Hovey, Mark and Strickland, Neil P.},
	year = {1999},
	doi = {10.1090/memo/0666},
}

@article{hovey-strickland_2005a,
	title = {Comodules and {Landweber} exact homology theories},
	volume = {192},
	issn = {0001-8708},
	url = {https://www.sciencedirect.com/science/article/pii/S0001870804001288},
	doi = {10.1016/j.aim.2004.04.011},
	number = {2},
	journal = {Advances in Mathematics},
	author = {Hovey, Mark and Strickland, Neil P.},
	year = {2005},
	pages = {427--456},
}

@book{brodmann-sharp_1998,
	address = {Cambridge},
	edition = {2},
	series = {Cambridge {Studies} in {Advanced} {Mathematics}},
	title = {Local {Cohomology}: {An} {Algebraic} {Introduction} with {Geometric} {Applications}},
	note = {ISBN: 978-0-521-51363-0},
	shorttitle = {Local {Cohomology}},
	url = {https://www.cambridge.org/core/books/local-cohomology/58A833CCD2D0F834644781AFB1351657},
	publisher = {Cambridge University Press},
	author = {Brodmann, Markus and Sharp, Rodney},
	year = {2012},
	doi = {10.1017/CBO9781139044059},
}

@article{bousfield_1979_localization,
	title = {The localization of spectra with respect to homology},
	volume = {18},
	issn = {0040-9383},
	url = {https://www.sciencedirect.com/science/article/pii/0040938379900181},
	doi = {10.1016/0040-9383(79)90018-1},
	number = {4},
	journal = {Topology},
	author = {Bousfield, Aldridge K.},
	year = {1979},
	pages = {257--281},
}

@article{bousfield_1985,
	title = {On the homotopy theory of {K}-local spectra at an odd prime},
	volume = {107},
	issn = {00029327},
	url = {https://www.jstor.org/stable/2374361?origin=crossref},
	doi = {10.2307/2374361},
	number = {4},
	journal = {American Journal of Mathematics},
	author = {Bousfield, Aldridge K.},
	year = {1985},
	pages = {895--932},
}

@article{franke_96,
	title = {Uniqueness for certain categories with an {A}dams {SS}},
	url = {},
	journal = {K-theory preprint archives},
	author = {Jens Franke},
  volume = {139},
  year = {1996},
  pages = {},
}

@article{schwede-schipley_02,
	title = {A uniqueness theorem for stable homotopy theory},
	volume = {239},
	issn = {1432-1823},
	url = {https://doi.org/10.1007/s002090100347},
	doi = {10.1007/s002090100347},
	number = {4},
	journal = {Mathematische Zeitschrift},
	author = {Schwede, Stefan and Shipley, Brooke},
	year = {2002},
	pages = {803--828},
}

@article{schwede_01,
	title = {The stable homotopy category has a unique model at the prime $2$},
	volume = {164},
	issn = {0001-8708},
	url = {https://www.sciencedirect.com/science/article/pii/S0001870801920092},
	doi = {10.1006/aima.2001.2009},
	number = {1},
	journal = {Advances in Mathematics},
	author = {Schwede, Stefan},
	year = {2001},
	pages = {24--40},
}

@article{schwede_07,
	title = {The stable homotopy category is rigid},
	volume = {166},
	issn = {0003-486X},
	doi = {10.4007/annals.2007.166.837},
	number = {3},
	journal = {Annals of Mathematics. Second Series},
	author = {Schwede, Stefan},
	year = {2007},
	pages = {837--863},
}

@article{patchkoria_2013,
	title = {On the algebraic classification of module spectra},
	volume = {12},
	issn = {1472-2747, 1472-2739},
	url = {https://projecteuclid.org/journals/algebraic-and-geometric-topology/volume-12/issue-4/On-the-algebraic-classification-of-module-spectra/10.2140/agt.2012.12.2329.full},
	doi = {10.2140/agt.2012.12.2329},
	number = {4},
	journal = {Algebraic \& Geometric Topology},
	author = {Patchkoria, Irakli},
	year = {2012},
	pages = {2329--2388},
}

@article{barthel-schlank-stapleton_2020,
	title = {Chromatic homotopy theory is asymptotically algebraic},
	volume = {220},
	issn = {1432-1297},
	url = {https://doi.org/10.1007/s00222-019-00943-9},
	doi = {10.1007/s00222-019-00943-9},
	number = {3},
	journal = {Inventiones Mathematicae},
	author = {Barthel, Tobias and Schlank, Tomer and Stapleton, Nathaniel},
	year = {2020},
	pages = {737--845},
}

@article{barthel-schlank-stapleton_2021,
	title = {Monochromatic homotopy theory is asymptotically algebraic},
	volume = {393},
	issn = {0001-8708},
	url = {https://www.sciencedirect.com/science/article/pii/S0001870821004382},
	doi = {10.1016/j.aim.2021.107999},
	journal = {Advances in Mathematics},
	author = {Barthel, Tobias and Schlank, Tomer and Stapleton, Nathaniel},
	year = {2021},
}

@book{ravenel_86,
	title = {Complex cobordismc and stable homotopy groups of spheres},
  edition = {2},
	note = {ISBN: 978-1-4704-7293-1},
	publisher = {AMS Chelsea Publishing},
	author = {Ravenel, Douglas C.},
	year = {1986},
}

@book{ravenel_92,
	title = {Nilpotence and periodicity in stable homotopy theory},
	note = {ISBN: 978-0-691-08792-4},
	url = {https://www.jstor.org/stable/j.ctt1b9rzmg},
	publisher = {Princeton University Press},
	author = {Ravenel, Douglas C.},
	year = {1992},
}

@incollection{hovey_04,
	address = {Providence, RI},
	series = {Contemporary mathematics},
  volume = {346},
	title = {Homotopy theory of comodules over a Hopf algebroid},
	isbn = {0-8218-3285-9},
	url = {https://mathscinet.ams.org/mathscinet/relay-station?mr=2066503},
	booktitle = {Homotopy theory: relations with algebraic geometry, group cohomology, and algebraic {K}-theory},
	publisher = {American Mathematical Society},
	author = {Hovey, Mark},
	year = {2004},
	pages = {261--304},
}

@article{joyal_02,
	series = {Special volume celebrating the 70th birthday of professor {Max} {Kelly}},
	title = {Quasi-categories and {Kan} complexes},
	volume = {175},
	issn = {0022-4049},
	url = {https://www.sciencedirect.com/science/article/pii/S0022404902001354},
	doi = {10.1016/S0022-4049(02)00135-4},
	number = {1},
	journal = {Journal of Pure and Applied Algebra},
	author = {Joyal, André},
	year = {2002},
	pages = {207--222},
}

@article{greenlees-shipley_2013,
	title = {The {Cellularization} {Principle} for {Quillen} adjunctions},
	volume = {15},
	issn = {1532-0081},
	url = {https://www.intlpress.com/site/pub/pages/journals/items/hha/content/vols/0015/0002/a011/abstract.php},
	doi = {10.4310/HHA.2013.v15.n2.a11},
	number = {2},
	journal = {Homology, Homotopy and Applications},
	author = {Greenlees, John P. C. and Shipley, Brooke},
	year = {2013},
	pages = {173--184},
}

@book{hovey-palmiery-strickland_97,
	title = {Axiomatic {Stable} {Homotopy} {Theory}},
	note = {ISBN: 978-0-8218-0624-1},
	publisher = {American Mathematical Society},
	author = {Hovey, Mark and Palmieri, John H. and Strickland, Neil P.},
	year = {1997},
}

@book{lurie_09,
	title = {Higher topos theory},
	note = {ISBN: 978-0-691-14049-0},
	url = {https://www.jstor.org/stable/j.ctt7s47v},
	publisher = {Princeton University Press},
	author = {Lurie, Jacob},
	year = {2009},
}

@article{ishak_19,
	title = {Rigidity of the $K(1)$-local stable homotopy category},
	volume = {21},
	issn = {1532-0081},
	url = {https://www.intlpress.com/site/pub/pages/journals/items/hha/content/vols/0021/0002/a014/abstract.php},
	doi = {10.4310/HHA.2019.v21.n2.a14},
	number = {2},
	journal = {Homology, Homotopy and Applications},
	author = {Ishak, Jocelyne},
	year = {2019},
	pages = {261--278},
}

@article{roitzheim_07,
	title = {Rigidity and exotic models for the $K$-local stable homotopy category},
	volume = {11},
	issn = {1465-3060, 1364-0380},
	url = {https://projecteuclid.org/journals/geometry-and-topology/volume-11/issue-4/Rigidity-and-exotic-models-for-the-Klocal-stable-homotopy-category/10.2140/gt.2007.11.1855.full},
	doi = {10.2140/gt.2007.11.1855},
	number = {4},
	journal = {Geometry \& Topology},
	author = {Roitzheim, Constanze},
	year = {2007},
	pages = {1855--1886},
}

\textbf{Torgeir Aamb\o:} Department of Mathematical Sciences, Norwegian University of Science and Technology, Trondheim \\
\textbf{Email address:} torgeir.aambo@ntnu.no \\
\textbf{Website:} \href{https://folk.ntnu.no/torgeaam/ }{https://folk.ntnu.no/torgeaam/}

\end{document}